\documentclass[11pt,a4paper]{article}

\usepackage{calc}
\usepackage{hyperref}
\usepackage[a4paper,tmargin= 1.8cm,bmargin= 1.8cm,rmargin=2.5cm,lmargin=2.5cm]{geometry}
\usepackage{cancel}
\usepackage{makeidx}
\makeindex

\usepackage[all]{xy}
\usepackage[centertags]{amsmath}
\usepackage{latexsym}
\usepackage{amsfonts}
\usepackage{amssymb}
\usepackage{amsthm}
\usepackage{mathpazo}

\usepackage [dvips]{epsfig}
\usepackage{graphicx}
\usepackage{newlfont}
\usepackage{color}

\usepackage[latin1]{inputenc}

\usepackage[french,english]{babel}
\usepackage{graphicx}
\usepackage{t1enc}

\usepackage[normalem]{ulem}
\usepackage{pifont}
\usepackage{lscape}
\usepackage{tabularx}
\usepackage{multirow}%
\usepackage{textcomp}
\usepackage{multicol}

\newcommand{\N}{\ensuremath{\mathbb{N}}}

\newcommand{\Z}{\ensuremath{\mathbb{Z}}}

\usepackage{url}
\usepackage{amssymb}

\usepackage[all]{xy}
\usepackage[centertags]{amsmath}
\usepackage{latexsym}
\usepackage{amsfonts}
\usepackage{amsthm}

\def \A {\mathfrak{A}}

\newcommand{\Dpq}{D_{p,q}}

\newcommand{\Epq}{\textsl{E}_{p,q}}
\newcommand{\epq}{\textsl{e}_{p,q}}

\newcommand{\pqbinomial}[4]{\mbox{$
\biggl[\!
\begin{array}{c}
#1\\
 #2
\end{array}\!\biggr]_{
\!{#3,#4}} $} }

\def \be {\begin{equation}}
\def \ee {\end{equation}}

\def \btb {\begin{tablen}}
\def \etb {\end{tablen}}

\def \bea {\begin{eqnarray}}
\def \eea {\end{eqnarray}}
\def \bean {\begin{eqnarray*}}
\def \eean {\end{eqnarray*}}

\newtheorem{theorem}{Theorem}
\newtheorem{remark}[theorem]{Remark}
\newtheorem{proposition}[theorem]{Proposition}
\newtheorem{definition}[theorem]{Definition}
\newtheorem{corollary}[theorem]{Corollary}

\def \be {\begin{equation}}
\def \ee {\end{equation}}

\def \btb {\begin{tablen}}
\def \etb {\end{tablen}}

\def \bea {\begin{eqnarray}}
\def \eea {\end{eqnarray}}
\def \bean {\begin{eqnarray*}}
\def \eean {\end{eqnarray*}}

\makeatletter

\title{On $(p,q)$-Appell Polynomials }

\author{{\textbf{P. Njionou Sadjang\footnote{Faculty of Industrial Engineering, University of Douala, Cameroon}\;\footnote{pnjionou@yahoo.fr}} }}

\begin{document}

\maketitle







%

\begin{abstract}
We introduce polynomial sets of $(p,q)$-Appell type and give some of their characterizations. The algebraic properties of the set of all polynomial sequences of $(p,q)$-Appell type are studied. Next, we give a recurrence relation and a $(p,q)$-difference equation for those polynomials. Finally, some examples of polynomial sequences of $(p,q)$-Appell type are given, particularly, a set of $(p,q)$-Hermite polynomials  is given and their three-term recurrence relation and a second order homogeneous $(p,q)$-difference equation are provided.     
\end{abstract}

%
%



\section{Introduction}
\label{}


\noindent Let $P_n(x)$, $n=0,\, 1,\,2,\,\cdots$ be a polynomial set, i.e. a sequence of polynomials with $P_n(x)$ of exact degre $n$. Assume that $\dfrac{dP_n(x)}{dx}=P'_n(x)=nP_{n-1}(x)$ for $n=0,\, 1,\,2,\,\cdots$. Such polynomial sets are called Appell sets and received considerable attention since P. Appell \cite{appell1880} introduced them in 1880. \\

\noindent Let $q$ and $p$ be two arbitrary real or complex numbers and define the $(p,q)$-derivative \cite{njionou2013} of a function $f(x)$ by means of
\begin{equation}
 \Dpq f(x)=\dfrac{f(px)-f(qx)}{(p-q)x},
\end{equation}
which furnishes a generalization of the so-called $q$-derivative (or Hahn derivative)
\begin{equation}
 D_q f(x)=\dfrac{f(x)-f(qx)}{(1-q)x},
\end{equation}
which itself is a generalization of the classical differential operator $\dfrac{d}{dx}$. \\

\noindent The purpose of this paper is to study the class of polynomial sequences $\{P_n(x)\}$ which satisfy 
\begin{equation}\label{eq1}
 D_qP_n(x)=[n]_{p,q}P_{n-1}(px),\quad n=0,\; 1,\; 2,\;3,\;\cdots
\end{equation}
where $[n]_{p,q}$ is defined bellow. We note that when $p=1$,  \eqref{eq1} reduces to $D_q P_n(x)=[n]_qP_{n-1}(x)$ so that we may think of $(p,q)$-Appell sets as a generalization of $q$-Appell sets (see \cite{al-salam}).

\section{Preliminary results and definitions}

\subsection{$(p,q)$-number, $(p,q)$-factorial, $(p,q)$-binomial coefficients, $(p,q)$-power.}
\noindent Let us introduce the following notations (see  \cite{JS2006,  JR2010, njionou2013})
\begin{equation*}\label{pqnumber}
[n]_{p,q}=\frac{p^n-q^n}{p-q},\quad 0<q<p
\end{equation*}
for any positive integer $n\in\N$. The twin-basic number is a natural generalization of the $q$-number, that is
\begin{equation*}
\lim\limits_{p\to 1}[n]_{p,q}=[n]_q.
\end{equation*}

\noindent The $(p,q)$-factorial is defined by  (see \cite{JR2010, njionou2013})
\begin{equation*}
[n]_{p,q}!=\prod_{k=1}^{n}[k]_{p,q}!,\quad n\geq 1,\quad [0]_{p,q}!=1.
\end{equation*}
\noindent Let us introduce also the so-called $(p,q)$-binomial coefficients
\begin{equation*}\label{pqbin}
\pqbinomial{n}{k}{p}{q}=\dfrac{[n]_{p,q}!}{[k]_{p,q}![n-k]_{p,q}!}, \quad 0\leq k\leq n.
\end{equation*}
 Note that as $p\to 1$, the $(p,q)$-binomial coefficients reduce to the $q$-binomial coefficients.\\
 It is clear by definition that 
 \begin{equation*}\label{binomial1}
 \pqbinomial{n}{k}{p}{q}=\pqbinomial{n}{n-k}{p}{q}.
 \end{equation*}
 
\noindent  Let us introduce also the so-called falling and raising $(p,q)$-powers, respectively \cite{njionou2013}
\begin{eqnarray*}
 (x\ominus a)_{p,q}^n&=&(x-a)(px-aq)\cdots (xp^{n-1}-aq^{n-1}),\\
 (x\oplus a)_{p,q}^n&=&(x+a)(px+aq)\cdots (xp^{n-1}+aq^{n-1}).
\end{eqnarray*}

\noindent These definitions are extended to 
\begin{eqnarray*}\label{infinitpq}
(a\ominus b)_{p,q}^{\infty}=\prod_{k=0}^{\infty}(ap^k-q^kb),\\
(a\oplus b)_{p,q}^{\infty}=\prod_{k=0}^{\infty}(ap^k+q^kb),
\end{eqnarray*}
where convergence is required.

\subsection{The $(p,q)$-derivative and the $(p,q)$-integral}

\noindent Let $f$ be a function defined on the set of the complex numbers.

\begin{definition}[See \cite{njionou2013}]
The $(p,q)$-derivative of  the function $f$  is defined as  
\begin{equation*}
\Dpq f(x)=\dfrac{f(px)-f(qx)}{(p-q)x},\quad x\neq0,
\end{equation*}
and $(\Dpq f)(0)=f'(0)$,
provided that $f$ is differentiable at $0$.
\end{definition}

\begin{proposition}[See \cite{njionou2013}]
The $(p,q)$-derivative fulfils the following product and quotient rules
\begin{eqnarray*}
\Dpq (f(x)g(x))&=& f(px)\Dpq g(x)+g(qx)\Dpq f(x),\label{productrule2}\\
\Dpq (f(x)g(x))&=& g(px)\Dpq f(x)+f(qx)\Dpq g(x).\label{productrule3}
\end{eqnarray*}
\end{proposition}

\subsection{$(p,q)$-exponential and $(p,q)$-trigonometric functions.}

\noindent As in the $q$-case, there are many definitions of the $(p,q)$-exponential function. The following two $(p,q)$-analogues of the exponential function (see \cite{JS2006}) will be frequently used throughout this paper:
\begin{eqnarray}
 \epq(z)&=&\sum_{n=0}^{\infty}\dfrac{p^{\binom{n}{2}}}{[n]_{p,q}!}z^n,\label{pqexp}\\
 E_{p,q}(z)&=&\sum_{n=0}^{\infty}\dfrac{q^{\binom{n}{2}}}{[n]_{p,q}!}z^n.\label{bigpqexp}
\end{eqnarray}
\begin{proposition} The following equation applies: 
\begin{equation}\label{invpqexpo}
\epq(x)E_{p,q}(-x)=1. 
\end{equation}
\end{proposition}

\begin{proof}
The result is proved in \cite{JS2006} using $(p,q)$-hypergeometric series. We provide here a direct proof. 
From \eqref{pqexp} and \eqref{bigpqexp}, and the general identity (which is a direct consequence of the Cauchy product)
\begin{equation}\label{gen-identity}
\left(\sum_{n=0}^{\infty}a_n\dfrac{t^n}{[n]_{p,q}!}\right)\left(\sum_{n=0}^{\infty}b_n\dfrac{t^n}{[n]_{p,q}!}\right)=\sum_{n=0}^{\infty}\left(\sum_{k=0}^{n}\pqbinomial{n}{k}{p}{q}a_kb_{n-k}\right)\dfrac{t^n}{[n]_{p,q}!},  
\end{equation}
it follows that 
\begin{eqnarray*}
\epq(x)E_{p,q}(-x)&=& \left(\sum_{n=0}^{\infty}\dfrac{p^{\binom{n}{2}}}{[n]_{p,q}!}x^n\right)\left(\sum_{n=0}^{\infty}\dfrac{q^{\binom{n}{2}}}{[n]_{p,q}!}(-x)^n\right)\\
&=& \sum_{n=0}^{\infty}\left(\sum_{k=0}^{n}\pqbinomial{n}{k}{p}{q}(-1)^kq^{\binom{k}{2}}p^{\binom{n-k}{2}}   \right)\dfrac{x^n}{[n]_{p,q}!}
\end{eqnarray*}
It remains to prove that
\[\sum_{k=0}^{n}\pqbinomial{n}{k}{p}{q}(-1)^kq^{\binom{k}{2}}p^{\binom{n-k}{2}}=\delta_{n,0}.\]
It is not difficult to prove that for every polynomial $f_n(x)$ of degree $n$, the Taylor formula 
\[f_n(x)=\sum_{k=0}^{n}\dfrac{(\Dpq^k f)(0)}{[k]_{p,q}!}x^n,\]
holds true. Applying this formula to $f_n(x)=(a\ominus x)_{p,q}^n$, it follows that  
\[(a\ominus x)_{p,q}^n=\sum_{k=0}^{n}\pqbinomial{n}{k}{p}{q}q^{\binom{k}{2}}p^{\binom{n-k}{2}}(-x)^ka^{n-k}.\]
Taking finally $x=a=1$, the result follows. 
\end{proof}

\noindent The next proposition gives the $n$-th derivative of the $(p,q)$-exponential functions.  
\begin{proposition}
Let $n$ be a nonnegative integer, $\lambda$ a complex number, then  the following equations hold 
\begin{eqnarray}
 \Dpq^n \epq(\lambda x)=\lambda^n p^{\binom{n}{2}} \epq(\lambda p^nx),\label{nth-deriv-pqExp}\\ 
 \Dpq^n E_{p,q}(\lambda x)=\lambda^n q^{\binom{n}{2}}\lambda E_{p,q}(\lambda q^nx). \label{nth-deriv-pqexp}
\end{eqnarray}
\end{proposition}

\begin{proof}
The proof follows by induction from the definitions of the $(p,q)$-exponentials and the $(p,q)$-derivative.
\end{proof}


\section{$(p,q)$-Appell polynomials}


\begin{definition}
A polynomial sequence $\{f_n(x)\}_{n=0}^{\infty}$ is called a $(p,q)$-Appell polynomial sequence if and only if 
\begin{equation}
\Dpq f_{n+1}(x)=[n+1]_{p,q}f_{n}(px),\;\;n\geq 0.
\end{equation}
\end{definition}
\noindent It is not difficult to see that the polynomial sequence $\{f_n(x)\}_{n=0}^{\infty}$ with  $f_n(x)=(x\ominus a)_{p,q}^n$ is a $(p,q)$-Appell polynomial sequence since (see \cite{njionou2013})
\[\Dpq (x\ominus a)_{p,q}^n=[n]_{p,q}(px\ominus a)_{p,q}^{n-1}, n\geq 1.\]

\begin{remark}
Note that when $p=1$, we obtain the classical $q$-Appell polynomial sequences known in the literature \cite{al-salam}. When $q=1$, we obtain the new basic Appell polynomial sequences of type II introduced and extensively studied in \cite{njionou2018-2}. 
\end{remark}

Next, we give several characterizations of $(p,q)$-Appell polynomial sequences.
\begin{theorem}\label{theo-characterization}
Let $\{f_n(x)\}_{n=0}^{\infty}$ be a sequence of polynomials. Then the following are all equivalent: 
\begin{enumerate}
   \item $\{f_n(x)\}_{n=0}^{\infty}$ is a $(p,q)$-Appell polynomial sequence.
   \item There exists a sequence $(a_k)_{k\geq 0}$, independent of $n$, with $a_0\neq 0$ and such that 
   \[ f_n(x)=\sum_{k=0}^{n}\pqbinomial{n}{k}{p}{q}p^{\binom{n-k}{2}}a_{k}x^{n-k}.\]
   \item $\{f_n(x)\}_{n=0}^{\infty}$ is generated by 
   \[A(t)\epq(xt)=\sum_{n=0}^{\infty}f_n(x)\dfrac{t^n}{[n]_{p,q}!},\]
   with the determining function 
   \begin{equation}
     A(t)=\sum_{n=0}^{\infty}a_n\dfrac{t^n}{[n]_{p,q}!}.
   \end{equation}
   \item There exists a sequence $(a_k)_{k\geq 0}$, independent of $n$ with $a_0\neq 0$ and such that
   \[ f_n(x)=\left(\sum_{k=0}^{\infty}\dfrac{a_k p^{\binom{n-k}{2}}}{[k]_{p,q}!}\Dpq^{k}\right)x^n.\]
\end{enumerate}
\end{theorem}

\begin{proof}
First, we prove that $(1)\implies (2)\implies (3)\implies (1)$.
\begin{description}
   \item[$(1)\implies (2)$.] Since $\{f_n(x)\}_{n=0}^{\infty}$ is a polynomial set, it is possible to write 
   \begin{equation}\label{exp1}
      f_n(x)=\sum_{k=0}^{n}a_{n,k}\pqbinomial{n}{k}{p}{q}p^{\binom{n-k}{2}}x^{n-k},\quad n=1,2,\ldots,
   \end{equation}
   where the coefficients $a_{n,k}$ depend on $n$ and $k$ and $a_{n,0}\neq 0$. We need to prove that these coefficients are independent of $n$. By  applying the operator $\Dpq$ to each member of \eqref{exp1} and taking into account that $\{f_n(x)\}_{n=0}^{\infty}$ is a $(p,q)$-Appell polynomial set, we obtain 
   \begin{equation}\label{exp2}
      f_{n-1}(px)=\sum_{k=0}^{n-1}a_{n,k}\pqbinomial{n-1}{k}{p}{q}p^{\binom{n-1-k}{2}}(px)^{n-1-k}, \quad n=1,2,\ldots,
   \end{equation}
   since $\Dpq x^{0}=0$. Shifting the index $n\to n+1$ in \eqref{exp2} and making the substitution $x\to xp^{-1}$, we get
   \begin{equation}\label{exp3}
      f_n(x)=\sum_{k=0}^{n}a_{n+1,k}\pqbinomial{n}{k}{p}{q}p^{\binom{n-k}{2}}x^{n-k},\quad n=0,1,\ldots,
   \end{equation}
   Comparing \eqref{exp1} and \eqref{exp3}, we have $a_{n+1,k}=a_{n,k}$ for all $k$ and $n$, and therefore $a_{n,k}=a_k$ is independent of $n$.  
   \item[$(2)\implies (3)$.] From $(2)$, and the identity \eqref{gen-identity}, we have 
   \begin{eqnarray*}
   \sum_{n=0}^{\infty}f_n(x)\dfrac{t^n}{[n]_{p,q}!}&=& \sum_{n=0}^{\infty}\left(\sum_{k=0}^{n}\pqbinomial{n}{k}{p}{q}p^{\binom{n-k}{2}}a_{k}x^{n-k}\right)\dfrac{t^n}{[n]_{p,q}!}\\
   &=& \left(\sum_{n=0}^{\infty}a_n \dfrac{t^n}{[n]_{p,q}!}\right)\left(\sum_{n=0}^{\infty}\dfrac{p^{\binom{n}{2}}}{[n]_{p,q}!}(xt)^n\right)\\
   &=& A(t)\epq(xt). 
   \end{eqnarray*}
   \item[$(3)\implies (1)$.] Assume that $\{f_n(x)\}_{n=0}^{\infty}$ is generated by 
   \[A(t)\epq(xt)=\sum_{n=0}^{\infty}f_n(x)\dfrac{t^n}{[n]_{p,q}!}.\] 
   Then, applying the operator $\Dpq$ (with respect to the variable $x$) to each side of this equation , we get
   \begin{eqnarray*}
    tA(t)\epq(pxt)&=& \sum_{n=0}^{\infty}\Dpq f_n(x)\dfrac{t^n}{[n]_{p,q}!}.
   \end{eqnarray*}
   Moreover, we have 
   \begin{eqnarray*}
     tA(t)\epq(pxt)&=& \sum_{n=0}^{\infty}f_{n}(px)\dfrac{t^{n+1}}{[n]_{p,q}!}=\sum_{n=0}^{\infty}[n]_{p,q}f_{n-1}(px)\dfrac{t^n}{[n]_{p,q}!}.
   \end{eqnarray*}
   By comparing the coefficients of $t^n$, we obtain $(1)$. 
\end{description}
Next, $(2)\iff (4)$ is obvious since $\Dpq^k t^n=0$ for $k>n$. 
This ends the proof of the theorem.
\end{proof}

\section{Algebraic structure}

\noindent  We denote a given polynomial set $\{f_n(x)\}_{n=0}^{\infty}$ by a single symbol $f$ and refer to $f_n(x)$ as the $n$-th component of $f$. We define (as done in \cite{appell1880, sheffer1931}) on the set $\mathcal{P}$ of all polynomials sequences the following three operations $+$, $.$ and $*$. The first one is given by the rule that $f+g$ is the polynomial sequence whose $nth$ component is $f_n(x)+g_n(x)$ provided that the degree of $f_n(x)+g_n(x)$ is exactly $n$. On the other hand, if $f$ and $g$ are two sets whose $nth$ components are, respectively, 
\[f_n(x)=\sum_{k=0}^n\alpha(n,k)x^k,\quad g_n(x)=\sum_{k=0}^n\beta(n,k)x^k,\]
then $f*g$ is the polynomial set whose $nth$ component is given by 
\[ (f*g)_n(x)=\sum_{k=0}^n\alpha(n,k)p^{-\binom{k}{2}}g_k(x).\]
If $\lambda$ is a real or complex number, then $\lambda f$ is defined as the polynomial sequence whose $nth$ component is $\lambda f_n(x)$. We obviously have 
\begin{eqnarray*}
f+g=g+f\quad \textrm{for all}\quad f,g\in \mathcal{P},\\
\lambda f*g=(f*\lambda g)=\lambda(f*g).
\end{eqnarray*}
Clearly, the operation $*$ is not commutative (see \cite{sheffer1931}). One commutative subclass is the set $\A$ of all Appell polynomials (see \cite{appell1880}). 

In what follows, $\A(p,q)$ denotes the class of all $(p,q$)-Appell sets. 

In $\A(p,q)$ the identity element (with respect to $*$) is the $(p,q)$-Appell set $I=\left\{p^{\binom{n}{2}}x^n\right\}$. Note that $I$ has the determining function $A(t)=1$. This is due to identity \eqref{pqexp}. The following theorem is easy to prove. 

\begin{theorem}
Let $f,\, g,\, h\in\A(p,q)$ with the determining functions $A(t)$, $B(t)$ and $C(t)$, respectively. Then 
\begin{enumerate}
   \item $f+g\in\A(p,q)$ if $A(0)+B(0)\neq 0$,
   \item $f+g$ belongs to the determining function $A(t)+B(t)$,
   \item $f+(g+h)=(f+g)+h$.
\end{enumerate}
\end{theorem} 
\noindent The next theorem is less obvious. 

\begin{theorem}\label{theo-group}
If $f,\, g,\, h\in\A(p,q)$ with the determining functions $A(t)$, $B(t)$ and $C(t)$, respectively, then
\begin{enumerate}
    \item $f*g\in\A(p,q)$
    \item $f*g=g*f$,
    \item $f*g$ belongs to the determining function $A(t)B(t)$,
    \item $f*(g*h)=(f*g)*h$. 
\end{enumerate}
\end{theorem}
\begin{proof}
It is enough to prove the first part of the theorem. The rest follows directly. \\
According to Theorem \ref{theo-characterization}, we may put 
\[ f_n(x)=\sum_{k=0}^{n}\pqbinomial{n}{k}{p}{q}p^{\binom{n-k}{2}}a_{k}x^{n-k}=\sum_{k=0}^{n}\pqbinomial{n}{k}{p}{q}p^{\binom{k}{2}}a_{n-k}x^{k}\]
so that 
 \[A(t)=\sum_{n=0}^{\infty}a_n\dfrac{t^n}{[n]_{p,q}!}.\]
 Hence
 \begin{eqnarray*}
 \sum_{n=0}^{\infty}(f*g)_n(x)\dfrac{t^n}{[n]_{p,q}!}&=& \sum_{n=0}^{\infty}\left(\sum_{k=0}^n\pqbinomial{n}{k}{p}{q}a_{n-k}g_k(x)  \right) \dfrac{t^n}{[n]_{p,q}!}\\
 &=& \left(\sum_{n=0}^{\infty}a_n\dfrac{t^n}{[n]_{p,q}!}\right)\left(\sum_{n=0}^{\infty}g_n(x)\dfrac{t^n}{[n]_{p,q}!}\right)\\
 &=& A(t)B(t)\epq(xt).
 \end{eqnarray*}
 This ends the proof of the theorem. 
\end{proof}

\begin{corollary}\label{coro-inv}
Let $f\in\A(p,q)$, then $f$ has an inverse with respect to $*$, i.e. there is a set $g\in\A(p,q)$ such that 
\[f*g=g*f=I.\]
\end{corollary}
Indeed $g$ belongs to the determining function $(A(t))^{-1}$ where $A(t)$ is the determining function for $f$. 

In view of Corollary \ref{coro-inv} we shall denote this element $g$ by $f^{-1}$. We are further motivated by Theorem \ref{theo-group} and its corollary to define $f^{0}=I$, $f^{n}=f*(f^{n-1})$ where $n$ is a non-negative integer, and $f^{-n}=f^{-1}*(f^{-n+1})$. We note that we have proved that the system $(\A(p,q),*)$ is a commutative group. In particular this leads to the fact that if 
\[f * g=h\]
and if any two of the elements $f,\, g\, h$ are $(p,q)$-Appell then the third is also $(p,q)$-Appell. 

\begin{proposition}
If $f$ is a $(p,q)$-Appell sequence with the determining function $A(t)$, and if we set
\[A^{-1}(t)=\sum_{n=0}^{\infty}b_n\dfrac{t^n}{[n]_{p,q}!}\]
then 
\[x^n=p^{-\binom{n}{2}}\sum_{k=0}^{n}\pqbinomial{n}{k}{p}{q}b_kf_{n-k}(x).\]
\end{proposition}

\begin{proof}
Since $f$ is a $(p,q)$-Appell sequence, we have 
\begin{eqnarray*}
   \sum_{n=0}^{\infty}p^{\binom{n}{2}}x^n\dfrac{t^n}{[n]_{p,q}!}&=& (A(t))^{-1}A(t)\epq(xt)\\
   &=& \left(\sum_{n=0}^{\infty}b_n\dfrac{t^n}{[n]_{p,q}!}\right)\left(\sum_{n=0}^{\infty}f_n(x)\dfrac{t^n}{[n]_{p,q}!}\right)\\
   &=& \sum_{n=0}^{\infty}\left(\sum_{k=0}^{n}\pqbinomial{n}{k}{p}{q}b_kf_{n-k}(x)\right)\dfrac{t^n}{[n]_{p,q}!}.
\end{eqnarray*}
The result follows by comparing the coefficients of $t^n$. 
\end{proof}

%
%
%
%
%

%

\section{$(p,q)$-difference and $(p,q)$-recurrence relations for $(p,q)$-Appell polynomials}

\noindent In this section, we derive a recurrence relation and a $(p,q)$-difference equation for the $(p,q)$-Appell polynomials.

\begin{theorem}\label{recursion1}
Let $\{f_n(x)\}_{n=0}^{\infty}$ be the $(p,q)$-Appell polynomial sequences generated by 
\[\mathbf{A}(x,t)= A(t)\epq(xt)=\sum_{n=0}^{\infty}f_n(x)\dfrac{t^n}{[n]_{p,q}!}.\]
Then the following linear homogeneous recurrence relation  holds true: 
\begin{eqnarray}\label{recursion-formula1}
f_n(px/q)=\dfrac{1}{[n]_{p,q}}\sum_{k=0}^{n}\pqbinomial{n}{k}{p}{q}\alpha_kf_{n-k}(x)+{p^n}{q^{-1}}{x}f_{n-1}(x).
\end{eqnarray}
where 
\begin{equation}\label{assumption1}
t\dfrac{\Dpq^{\{t\}}A(t)}{A(pt)}=\sum_{n=0}^{\infty}\alpha_n\dfrac{t^n}{[n]_{p,q}!}.
\end{equation}
\end{theorem}

\begin{proof}
Applying the $(p,q)$-derivative $\Dpq$ (with respect to the variable $t$) to each 
\[\mathbf{A}(px,t)= A(t)\epq(pxt)=\sum_{n=0}^{\infty}f_n(px)\dfrac{t^n}{[n]_{p,q}!}\]
and multiplying the obtained equation by $t$, we get the following equations
\[t\Dpq^{\{t\}}\mathbf{A}(px,t)=t\sum_{n=0}^{\infty}[n]_{p,q}f_n(px)\dfrac{t^{n-1}}{[n]_{p,q}!}=\sum_{n=0}^{\infty}[n]_qf_{n}(px)\dfrac{t^n}{[n]_{p,q}!};\]
\begin{eqnarray*}
t\Dpq^{\{t\}}\mathbf{A}(px,t)&=& t\left[\Dpq^{\{t\}}\left(A(t)\epq(pxt)\right)\right]\\
&=& t\left[ A(pt)\Dpq^{\{t\}}\epq(pxt)+ \epq(pqxt)\Dpq^{\{t\}}A(t)    \right]\\
&=& tpxA(pt)\epq(pqxt)+t\Dpq A(t)\epq(pqxt)\\
&=& A(pt)\epq(pqxt)\left(tpx+t\dfrac{\Dpq^{\{t\}}A(t)}{A(pt)}\right)\\
&=& \mathbf{A}(qx,pt)\left(tpx+t\dfrac{\Dpq^{\{t\}}A(t)}{A(pt)}\right)
\end{eqnarray*}
From the assumption \eqref{assumption1}, it follows that 
\begin{eqnarray*}
\sum_{n=0}^{\infty}[n]_qf_{n}(x)\dfrac{t^n}{[n]_{p,q}!}&=& \mathbf{A}(qx,pt)\left(tpx+t\dfrac{\Dpq^{\{t\}}A(t)}{A(pt)}\right)\\
&=& \left(\sum_{n=0}^{\infty}p^nf_n(qx)\dfrac{t^n}{[n]_{p,q}!}\right)\left(\sum_{n=0}^{\infty}\alpha_n\dfrac{t^n}{[n]_q!}+ tpx\right)\\
&=& \sum_{n=0}^{\infty}\left(\sum_{k=0}^n\pqbinomial{n}{k}{p}{q}\alpha_kp^{n-k}f_{n-k}(qx)\right)\dfrac{t^n}{[n]_{p,q}!}+x\sum_{n=0}^{\infty}p^{n+1}f_{n}(qx)\dfrac{t^{n+1}}{[n]_{p,q}!}\\
&=& \sum_{n=0}^{\infty}\left(\sum_{k=0}^n\pqbinomial{n}{k}{p}{q}\alpha_kp^{n-k}f_{n-k}(qx)\right)\dfrac{t^n}{[n]_{p,q}!}+x\sum_{n=0}^{\infty}[n]_{p,q}p^{n}f_{n-1}(qx)\dfrac{t^{n}}{[n]_{p,q}!}.
\end{eqnarray*}
Equating the coefficients of $t^n$ we obtain 
\[[n]_{p,q}f_{n}(px)=\sum_{k=0}^n\pqbinomial{n}{k}{p}{q}\alpha_kp^{n-k}f_{n-k}(qx)+[n]_{p,q}p^{n}xf_{n-1}(qx).  \]
Substituting $x$ by $xq^{-1}$ we obtain the result. 
\end{proof}

\begin{theorem}\label{pqdifference}
Let $\{f_n(x)\}_{n=0}^{\infty}$ be the $(p,q)$-Appell polynomial sequence generated by 
\[A(t)\epq(xt)=\sum_{n=0}^{\infty}f_n(x)\dfrac{t^n}{[n]_{p,q}!}.\]
Assume that 
\[t\dfrac{\Dpq^{\{t\}}A(t)}{A(pt)}=\sum_{n=0}^{\infty}\alpha_n\dfrac{t^n}{[n]_{p,q}!},\]
is valued around the point $t=0$. Then the $f_n$'s satisfy the $(p,q)$-difference equation 
\[\left(\sum_{k=0}^{n}\dfrac{\alpha_k}{[k]_{p,q}!}\mathcal{L}_p^{-k}\Dpq^{k}+{p^n}{q^{-1}}{x}\mathcal{L}^{-1}_p\Dpq\right)f_{n}(x)-[n]_{p,q}f_n(px/q)=0,\]
with 
\[\mathcal{L}_p^kf_n(x)=f_n(p^kx),\quad k\in\Z.\]
\end{theorem}

\begin{proof}
From Theorem \ref{recursion1}, we know that the $f_n$'s satisfy the recursion formula \eqref{recursion-formula1}. Since $\{f_n(x)\}_{n=0}^{\infty}$ is a $(p,q)$-Appell polynomial sequence, we have 
\[D_{p,q}^k f_n(x)=\dfrac{[n]_{p,q}!}{[n-k]_{p,q}!}f_{n-k}(p^kx),\quad 0\leq k\leq n.\]
It follows that 
\[f_{n-k}(x)=\dfrac{[n-k]_{p,q}!}{[n]_{p,q}!}\mathcal{L}_p^{-k}\Dpq^{k}f_{n}(x),\quad 0\leq k\leq n.\]
Then \eqref{recursion-formula1} becomes 
\begin{eqnarray*}
f_n(px/q)&=&\dfrac{1}{[n]_{p,q}}\sum_{k=0}^{n}\pqbinomial{n}{k}{p}{q}\alpha_k\dfrac{[n-k]_{p,q}!}{[n]_{p,q}!}\mathcal{L}_p^{-k}\Dpq^{k}f_{n}(x)+{p^n}{q^{-1}}{x}f_{n-1}(x)\\
&=& \dfrac{1}{[n]_{p,q}}\sum_{k=0}^{n}\dfrac{\alpha_k}{[k]_{p,q}!}\mathcal{L}_p^{-k}\Dpq^{k}f_{n}(x)+{p^n}{q^{-1}}{x}f_{n-1}(x)\\
&=& \dfrac{1}{[n]_{p,q}}\left(\sum_{k=0}^{n}\dfrac{\alpha_k}{[k]_{p,q}!}\mathcal{L}_p^{-k}\Dpq^{k}+{p^n}{q^{-1}}{x}\mathcal{L}^{-1}_p\Dpq\right)f_{n}(x)
\end{eqnarray*}
and the result follows
\end{proof}

\section{Some $(p,q)$-Appell polynomial sequences}

\noindent In this section, we give four examples of $(p,q)$-Appell polynomial sequences and prove some of their main structure relations. The bivariate $(p,q$-Bernoulli, the bivariate $(p,q)$-Euler and the bivariate $(p,q)$-Genocchy polynomials are introduced in \cite{duran} and some of their relevant properties are given. Without any lost of the generality, we will restrict ourselves to the case $y=0$. Also, we introduce a new generalization of the $(p,q)$-Hermite polynomials. 

\subsection{The $(p,q)$-Bernoulli polynomials}

\noindent The $(p,q)$-Bernoulli polynomials are $(p,q)$-Appell polynomials for the determining function $A(t)=\dfrac{t}{\epq(t)-1}$. Thus, the $(p,q)$-Bernoulli polynomials are defined by the generating function 
\[\dfrac{t}{\epq(t)-1}\epq(xt)=\sum_{n=0}^{\infty}\mathcal{B}_n(x;p,q)\dfrac{t^n}{[n]_{p,q}!}.\]
Let us define the $(p,q)$-Bernoulli numbers $\mathcal{B}_{n,p,q}$ by the generating function 
\[\dfrac{t}{\epq(t)-1}=\sum_{n=0}^{\infty}\mathcal{B}_{n,p,q}\dfrac{t^n}{[n]_{p,q}!}\] so that 
\[\mathcal{B}_{n}(0;p,q)=\mathcal{B}_{n,p,q},\quad (n\geq 0).\]

\begin{proposition}
The $(p,q)$-Bernoulli polynomials $\mathcal{B}_{n}(x;p,q)$ have the representation 
\begin{equation}
\mathcal{B}_n(x;p,q)=\sum_{n=0}^{n}\pqbinomial{n}{k}{p}{q}p^{\binom{n-k}{2}}\mathcal{B}_{k,p,q}x^{n-k}.
\end{equation}
\end{proposition}

\begin{proof}
The proof follows from Theorem \ref{theo-characterization}.
\end{proof}

\subsection{The $(p,q)$-Euler polynomials}

\noindent The $(p,q)$-Euler polynomials are $(p,q)$-Appell polynomials for the determining function $A(t)=\dfrac{2}{\epq(t)+1}$. Thus, the $(p,q)$-Euler polynomials are defined by the generating function 
\[\dfrac{2}{\epq(t)+1}\epq(xt)=\sum_{n=0}^{\infty}\mathcal{E}_n(x;p,q)\dfrac{t^n}{[n]_{p,q}!}.\]
Let us define the $(p,q)$-Euler numbers $\mathcal{E}_{n,p,q}$ by the generating function 
\[\dfrac{2}{\epq(t)+1}=\sum_{n=0}^{\infty}\mathcal{E}_{n,p,q}\dfrac{t^n}{[n]_{p,q}!}\] so that 
\[\mathcal{E}_{n}(0;p,q)=\mathcal{E}_{n,p,q},\quad (n\geq 0).\]

\begin{proposition}
The $(p,q)$-Euler polynomials $\mathcal{E}_{n}(x;p,q)$ have the representation 
\begin{equation}
\mathcal{E}_n(x;p,q)=\sum_{n=0}^{n}\pqbinomial{n}{k}{p}{q}p^{\binom{n-k}{2}}\mathcal{E}_{k,p,q}x^{n-k}.
\end{equation}
\end{proposition}

\begin{proof}
The proof follows from Theorem \ref{theo-characterization}.
\end{proof}

\subsection{The $(p,q)$-Genocchi polynomials}

\noindent The $(p,q)$-Genocchi polynomials are $(p,q)$-Appell polynomials for the determining function $A(t)=\dfrac{2t}{\epq(t)+1}$.  Thus, the $(p,q)$-Genocchi polynomials are defined by the generating function 
\[\dfrac{2t}{\epq(t)+1}\epq(xt)=\sum_{n=0}^{\infty}\mathcal{G}_n(x;p,q)\dfrac{t^n}{[n]_{p,q}!}.\]
Let us define the $(p,q)$-Genocchi numbers $\mathcal{G}_{n,p,q}$ by the generating function 
\[\dfrac{2t}{\epq(t)+1}=\sum_{n=0}^{\infty}\mathcal{G}_{n,p,q}\dfrac{t^n}{[n]_{p,q}!}\] so that 
\[\mathcal{G}_{n}(0;p,q)=\mathcal{G}_{n,p,q},\quad (n\geq 0).\]

\begin{proposition}
The $(p,q)$-Genocchi polynomials $\mathcal{G}_{n}(x;p,q)$ have the representation 
\begin{equation}
\mathcal{G}_n(x;p,q)=\sum_{n=0}^{n}\pqbinomial{n}{k}{p}{q}p^{\binom{n-k}{2}}\mathcal{G}_{k,p,q}x^{n-k}.
\end{equation}
\end{proposition}

\begin{proof}
The proof follows from Theorem \ref{theo-characterization}.
\end{proof}

\subsection{The $(p,q)$-Hermite polynomials}

\noindent In this section we construct  $(p,q)$-Hermite polynomials and give some of their properties. Also, we derive the three-term recurrence relation as well as the second-order differential equation satisfied by these polynomials. \\

\noindent We define  $(p,q)$-Hermite polynomials by means of the generating function 

\begin{equation}
 \mathbf{F}_{p,q}(x,t):=F_{p,q}(t)\epq(xt)=\sum_{n=0}^{\infty}H_n(x;p,q)\dfrac{t^n}{[n]_{p,q}!}.
\end{equation}
where 
\begin{equation}
 F_{p,q}(t)=\sum_{n=0}^{\infty}(-1)^np^{n(n-1)}\dfrac{t^{2n}}{[2n]_{p,q}!!},\;\;\textrm{with}\;\;\ [2n]_{p,q}!!=\prod_{k=1}^{n}[2k]_q,\quad [0]_{p,q}!!=1.
\end{equation}
It is clear that 
\begin{eqnarray*}
 \lim\limits_{p,q\to 1} \mathbf{F}_{p,q}(x,t)&=& e^{xt}\lim\limits_{p,q\to 1}\sum_{n=0}^{\infty}(-1)^np^{n(n-1)}\dfrac{t^{2n}}{[2n]_{p,q}!!}=e^{xt}\sum_{n=0}^{\infty}(-1)^n\dfrac{t^{2n}}{(2n)(2n-2)\cdots 2}\\
 &=&e^{xt}\sum_{n=0}^{\infty}(-1)^n\dfrac{t^{2n}}{2^nn!}=\exp\left(tx-\dfrac{t^2}{2}\right).
\end{eqnarray*}
Moreover, 
\begin{eqnarray*}
\Dpq^{\{t\}}{F}_{p,q}(t)=\sum_{n=1}^{\infty}(-1)^np^{n(n-1)}\dfrac{t^{2n-1}}{[2n-2]_{p,q}!!}=\sum_{n=0}^{\infty}(-1)^{n+1}p^{n(n-1)+2n}\dfrac{t^{2n+1}}{[2n]_{p,q}!!}=-tF_{p,q}(pt),
\end{eqnarray*}
Hence 
\[\dfrac{\Dpq^{\{t\}}{F}_{p,q}(t)}{F_{p,q}(pt)}=-t.\]

\begin{theorem}
The $(p,q)$-Hermite polynomials $H_n(x;p,q)$ have the following representation
\[H_{n}(x,p,q)=\sum_{k=0}^{\left[\frac{n}{2}\right]}\dfrac{(-1)^kp^{\binom{n-2k}{2}+k(k-1)}[n]_{p,q}!}{[2k]_{p,q}!![n-2k]_{p,q}!}x^{n-2k}.\]
\end{theorem}

\begin{proof}
Indeed, expanding the generating function $\mathbf{H}_{p,q}(x,t)$, we have 
\begin{eqnarray*}
\mathbf{H}_{p,q}(x,t)&=&\left(\sum_{k=0}^{\infty}(-1)^kp^{k(k-1)}\dfrac{t^{2k}}{[2k]_{p,q}!!}\right)\left(\sum_{n=0}^{\infty}p^{\binom{n}{2}}x^n\dfrac{t^n}{[n]_{p,q}!}\right)\\
&=& \sum_{n=0}^{\infty}\sum_{k=0}^{\infty}p^{\binom{n}{2}}x^n\dfrac{t^n}{[n]_{p,q}!}(-1)^kp^{k(k-1)}\dfrac{t^{2k}}{[2k]_{p,q}!!}.
\end{eqnarray*}
The result follows by using the series manipulation formula $(7)$ of Lemma 11 in \cite{rainville}. 
\end{proof}

\begin{theorem}
The following linear homogeneous recurrence relation for the $(p,q)$-Hermite polynomials holds true
\[ H_{n+1}(px,p,q)=p^{n+1}xH_{n}(qx,p,q)-p^{n-1}[n]_{p,q}H_{n-1}(qx,p,q),\quad (n\geq 1).\]
\end{theorem}

\begin{proof}
The result comes from Theorem \ref{recursion1} using the fact that  $\dfrac{\Dpq^{\{t\}}{F}_{p,q}(t)}{F_{p,q}(pt)}=-t$.
\end{proof}

\begin{theorem}
The $(p,q)$-Hermite polynomials $H_{n}(x;p,q)$ satisfy the $(p,q)$-difference equation 
\begin{equation}\label{pqdiff0}
\mathcal{L}_{p}^{-2}\Dpq^2 H_n(x;p,q)-p^2q^{-1}x\mathcal{L}_{p}^{-1}\Dpq H_n(x;p,q)+p^{2-n}[n]_{p,q}H_{n}(px/q)=0.
\end{equation}
\end{theorem}

\begin{proof}
The proof follows from Theorem \ref{pqdifference}.
\end{proof}

Note that as $p$ and $q$ tend to 1, Equation \eqref{pqdiff0} reduces to the second order differential equation satisfied by the Hermite polynomials.

\section*{Concluding remarks}

In this work we have introduced $(p,q)$-Appell sequences and have given several characterizations of these sequences. Also, by a suitable choose of the determining functions, we have recovered the $(p,q)$-Bernoulli and the $(p,q)$-Euler polynomials already given in \cite{duran}. It worth noting that we could set the problem of defining a new set of $(p,q)$-Appell sequences by changing the small $(p,q)$-exponential function $\epq$ by the big $(p,q)$-exponential function $\Epq$. But, this problem is useless since it is not difficult to see that $\epq=E_{q,p}$. Note also that the $(p,q)$-Appell defined here generalized both $q$-Appell functions of type I and of type II already found in the literature and can be viwed as a unified definition.  

\section*{Acknowledgements}
This work was supported by the Institute of Mathematics of the University of Kassel to whom I am very grateful. 



\end{document}